\documentclass{article}

\usepackage[utf8]{inputenc}
\usepackage{amsthm}
\usepackage{amssymb}
\usepackage{amsmath}
\usepackage{amsfonts}
\usepackage{geometry}
\usepackage[mathscr]{euscript}
\usepackage{tikz}
\usepackage{float}
\usepackage{url} 

\newtheorem{theorem}{Theorem}[section]
\newtheorem{lemma}[theorem]{Lemma}
\newtheorem{corollary}[theorem]{Corollary}
\newtheorem*{conjecture*}{Conjecture}

\theoremstyle{definition}
\newtheorem{definition}[theorem]{Definition}

\theoremstyle{definition}
\newtheorem{example}[]{Example}

\newcommand{\kemeny}{\mathscr{K}}

\title{Kemeny's constant for non-backtracking random walks}
\date{\today}
\date{}
\author{
    Jane Breen\footnote{Ontario Tech University, Oshawa, Ontario, Canada (jane.breen@ontariotechu.ca)},
    Nolan Faught\footnote{
        Brigham Young University, Provo, Utah, USA (faught3@gmail.com,
        mkempton@mathematics.byu.edu, adamarstk@yahoo.com, alice.oveson@gmail.com).
    }, 
    Cory Glover\footnote{Network Science Institute, Northeastern University, Boston, Massachusetts, USA (glover.co@northeastern.edu)}, \\
    Mark Kempton\footnotemark[2],
    Adam Knudson\footnotemark[2],
    Alice Oveson\footnotemark[2]
   
}

\begin{document}

\maketitle

\begin{abstract}
    Kemeny's constant for a connected graph $G$ is the expected time for a random walk to reach a randomly-chosen vertex $u$, regardless of the choice of the initial vertex. We extend the definition of Kemeny's constant to non-backtracking random walks and compare it to Kemeny's constant for simple random walks. We explore the relationship between these two parameters for several families of graphs and provide closed-form expressions for regular and biregular graphs. In nearly all cases, the non-backtracking variant yields the smaller Kemeny's constant.
\end{abstract}

\bigskip



\section{Introduction}
    A random walk on a graph $G = (V, E)$ is a Markov chain on $V$ that can model heat flow, games of
    chance, and solve combinatorial problems, among other applications. There has been growing interest
    in the behavior of non-backtracking random walks, which are Markov chains on $E$ that have many
    properties similar to simple random walks. The purpose of this work is to define Kemeny's constant
    for non-backtracking random walks, and to determine some of its properties. Kemeny's constant (which we
    introduce shortly) can be considered as the \emph{expected time to mixing} for the random walk on
    a graph, comparable to the \emph{mixing time} of a random walk on a graph.
    %% Rework: this doesn't belong to this paragraph, perhaps we should put it elsewhere?
    %% Alternatively, we can make this a separate paragraph and introduce our theme more in-depth
    It is well-known that, in many cases, the non-backtracking random walk on a graph has a
    faster mixing time than the simple random walk on a graph (see \cite{alon2007non,kempton2016non}).
    We explore similar angles here.  We prove in this paper that for regular and biregular graphs, the
    non-backtracking Kemeny's constant is smaller than the value of Kemeny's constant for the corresponding simple random walk (with only a few exceptions of small order).  This means that the non-backtracking random walk has a shorter expected time to mixing.  We likewise explore other families of graphs and find a significantly smaller Kemeny's constant using the non-backtracking random walk.
    
    A smaller Kemeny's constant indicates a short expected time to mixing, meaning that on average, hitting times are shorter.  Thus our results seem to indicate that, for many graphs, non-backtracking walks will have shorter average hitting times than their simple random walk counterparts. In applications using random-walk-based strategies, graphs with smaller Kemeny's constant tend to have more efficient performance (see for instance \cite{patel2015robotic}).  Thus our results suggest that in applications where random walks are used, and a small Kemeny's constant is desirable, use of a non-backtracking random walk may be more efficient.  Investigation into replacing simple random walks with non-backtracking walks in various applications is an  area of research receiving increased attention (see \cite{krzakala2013spectral,arrigo2019non}).  Our results suggest that this is a potentially important avenue for future research in applications where Kemeny's constant plays an important role.
    
    \section{Preliminaries}
    
    Throughout the paper, $G = (V, E)$ is a connected, undirected graph with $n = |V|$ vertices
    and $m = |E|$ edges. Edges are considered to be unordered pairs of distinct vertices $\{u, v\}$,
    and vertices $u$ and $v$ are said to be \emph{adjacent} if there is an edge $\{u, v\} \in E$. We
    also denote adjacency by $u\sim v$. If $u\sim v$, then $v$ is a \emph{neighbor} of $u$, and the
    number of neighbors of $u$ is called the \emph{degree} of the vertex $u$, denoted $\deg(u)$. The
    adjacency matrix of a graph of order $n$ is the $n\times n$ matrix $A=[a_{ij}]$ such that 
    \begin{equation*}
        a_{ij} = \begin{cases}
            1 &\text{if } i \sim j, \\
            0 &\text{otherwise.}
        \end{cases}
    \end{equation*}

    A discrete-time, time-homogeneous Markov chain on a finite state space $\{s_1, \ldots, s_n\}$
    models a system which occupies one of the states $s_i$ at any fixed time and transitions from one
    state to another in discrete time-steps. For any pair of states $s_i$, $s_j$, there is a
    transition probability $p_{ij}$ denoting the probability of transitioning to state
    $s_j$ in a single time-step, given that the system is currently in state $s_i$. Note that this
    probability does not depend on any states visited previously; this is called the Markov property.
    The transition probability matrix $P=[p_{ij}]$ encodes all information regarding the behavior of
    the Markov chain; the $(i,j)^{th}$ entry of $P^k$ denotes the probability of reaching state $s_j$
    in exactly $k$ steps, given that the system starts in state $s_i$. The $i^{th}$ row of the matrix
    $P^k$ thus gives the probability distribution across all states after $k$ time-steps, given that
    the system starts in state $s_i$. Under certain conditions on the matrix $P$, these probability
    distributions converge to a \emph{stationary distribution} independent of $i$; this stationary
    distribution vector $\pi$ is determined by the left eigenvector of $P$ corresponding to the
    spectral radius $1$ (which is in fact an eigenvalue due to Perron-Frobenius), normalized so that the entries sum to 1. The entry $\pi_i$ of the stationary distribution may be interpreted as the long-term probability that the Markov chain occupies the state $s_i$.

    The (simple) \emph{random walk} on a graph $G = (V, E)$ is a Markov chain whose states are the vertices of
    the graph, labelled in some order $v_1, \ldots, v_n$. If the random walk occupies $v_i$, the next state is chosen uniformly at random from
    the neighbors of $v_i$; that is, the random walker transitions to an adjacent vertex $v_j$ with
    probability $1/\deg(v_i)$. Thus, the transition matrix of this Markov chain is 
    \begin{equation*}
        P = D^{-1} A,
    \end{equation*}
    where $D$ is the diagonal matrix whose $i^{th}$ entry is the $\deg(v_i)$ and $A$ is the
    adjacency matrix of the graph. Note that the stationary distribution vector $\pi$ has $i^{th}$
    entry $\pi_i = \deg(v_i)/2m$; that is, the long-term probability of the random walker being
    on a vertex is proportional to the degree of that vertex.

    Given a Markov chain, we can also quantify its short-term behavior. The \emph{hitting time} or
    \emph{mean first passage time} from state $s_i$ to state $s_j$ of a Markov chain, denoted $m_{ij}$, is
    the expected  time it takes to reach state $s_j$, given that the system starts in state $s_i$. 
    %In particular, the value of the hitting time $T_j = \min\{t > 0: X_t =j\}$,
    %\begin{equation}
     %   m_{ij} = \mathbb E(T_j \mid X_0 = i).
    %\end{equation}
    A very interesting measure of the `average' short-term behavior of a Markov chain is known as
    \emph{Kemeny's constant}. Given an initial state $i$, define the quantity
    \begin{equation*}
        \kappa_i = \sum_{\substack{j=1\\j\neq i}}^n m_{ij} \pi_j.
    \end{equation*}
    where $m_{ij}$ is the mean first passage time from $s_i$ to $s_j$ and $\pi$ is the stationary
    distribution of the Markov chain. This quantity may be interpreted as the expected time to reach a
    randomly-chosen state $s_j$, given that we start in a fixed state $s_i$. Remarkably, this sum is
    independent of the initial state, and this quantity is known instead as \emph{Kemeny's constant},
    and denoted $\kemeny(P)$. Note that the above expression can be rewritten as follows:
    \begin{equation*}
        \kemeny(P) = \sum_{i=1}^n \sum_{\substack{j=1\\j\neq i}}^n \pi_im_{ij} \pi_j,
    \end{equation*}
    admitting the interpretation of $\kemeny(P)$ as the expected length of a trip between
    randomly-chosen states in the Markov chain (where `randomly-chosen' means with respect to the
    stationary distribution). In the case that $P$ represents the transition matrix for a simple random walk on a graph $G$, we can think of $\kemeny(P)$ as an inherent measure of the `connectedness' of the graph $G$, and denote this graph invariant by $\kemeny(G)$ instead.

    It is shown in \cite{levene2002kemeny} that Kemeny's constant can be expressed in terms of the
    eigenvalues of the transition matrix $P$.
    \begin{lemma}[\cite{levene2002kemeny}]\label{lem:kemeny_def}
        Given a Markov chain with transition matrix $P$ with eigenvalues $1 =\rho_1 > \rho_2
        \geq \rho_2 \geq \cdots \geq \rho_n$, then 
        \begin{equation*}
            \kemeny(P) = \sum_{i=2}^n \frac{1}{1 - \rho_i}.
        \end{equation*}
    \end{lemma}

    Hunter gives an interpretation of Kemeny's constant in \cite{hunter2006mixing} as the \emph{expected
    time to mixing} of a Markov chain. This is distinct from (but comparable to) the usual idea of \emph{mixing time} which describes the expected time taken for the Markov chain to become `close'
    to its stationary distribution. It is well-known that the spectral gap $1-\rho_2$, or the distance
    between the spectral radius of $P$ and its second-largest eigenvalue, bounds the rate of convergence
    of the Markov chain to the stationary distribution. If $1-\rho_2$ is small (i.e. $\rho_2$ is close
    to 1) the chain converges slowly. We note that from the eigenvalue expression for Kemeny's constant,
    it is clear that if there are eigenvalues close to $1$, that this will result in a large
    value of Kemeny's constant and thus indicate a chain for which the expected length of a random trip
    between states is relatively large, indicating poor mixing properties of the chain.

    Given a graph $G$, there is also an expression for $\kemeny(G)$ in terms of effective resistance that will at times be useful. We denote by $r(i,j)$ the \emph{effective resistance} between vertex $i$ and $j$,
    considering the graph as an electric circuit with each edge representing a unit resistor. This
    quantity is given by $r(i,j) = (e_i - e_j)^TL^\dag(e_i-e_j)$ where $e_i$ is the vector with a 1
    in the $i$-th position and zeros elsewhere and $L^\dag$ is the Moore-Penrose pseudoinverse of the
    graph Laplacian matrix (see \cite{bapat2010graphs}).
    
    \begin{lemma}[Corollary 1 of \cite{palacios2011broder}]\label{thm:kemeny}
        Suppose that $G = (V, E)$ is a simple connected graph, and let $R$ denote the matrix whose
        $(i, j)^{th}$ entry is the effective resistance between $i$ and $j$, $d$ the vector whose $i^{th}$
        entry is the degree of vertex $i$, and $m = |E|$. Kemeny's constant of the graph is related to
        the effective resistance by the identity
        \begin{equation*}
            \kemeny(G) = \frac{d^T Rd}{4m} = \frac{1}{4m}\sum_{i, j \in V} d_i d_j r(i, j).
        \end{equation*}
    \end{lemma}
    
    For certain graph families considered in this paper the following definitions will be helpful to
    deduce the value of Kemeny's constant for graphs with sparse structure. Note
    that \emph{moment} was first proposed for trees in \cite{ciardo2020kemeny}.
    
    \begin{definition}\label{def:moment}
        Let $G = (V, E)$ be a simple connected graph,  $R=[r(i,j)]$ the matrix of effective resistances in $G$ and $d$ the vector of vertex degrees. Let $e_v$ denote the vector with a 1 in the
        $v^{th}$ position and zeros elsewhere. The \emph{moment} of $v \in V$ is
        \begin{equation*}
            \mu(G, v) = d^T R e_v = \sum_{i\in V(G)} d_i r(i,v).
        \end{equation*}
    \end{definition}
    
    \begin{definition}
        Let $G_1, G_2$ be simple connected graphs, each with a vertex labelled $v$. The \emph{1-sum} $G = G_1\oplus_{v} G_2$ is the graph created by 
        taking copies of $G_1, G_2$, and identifying the copies of $v$. We often omit the subscript when the choice and/or labelling of
        vertices is clear. We say $G_1\oplus_vG_2$ has a \emph{1-separation}, and that $v$ is a
        \emph{1-separator} or \emph{cut vertex.}
    \end{definition}

    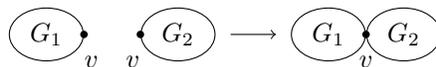
\begin{figure}[H]
        \centering
        \begin{tikzpicture}
            \tikzstyle{every node}=[circle, draw=none, fill=white, minimum width = 6pt, inner sep=1pt]
            % Apart
            \draw[] (1.75, 1)ellipse(14pt and 10pt);
            \draw[] (3.5, 1)ellipse(14pt and 10pt);
            \draw{
            (1.75,1)node[]{$G_1$}
            (3.5,1)node[]{$G_2$}
            (2.25,1)node[fill=black, minimum width = 2pt, label={[shift={(0.1,-.6)}]{$v$}}]{}
            (3,1)node[fill=black, minimum width = 2pt, label={[shift={(-0.1,-.6)}]{$v$}}]{}
            };
    
            %Arrow
            \draw[->] (4.2,1) -- (4.8,1);
    
            % Together
            \draw[] (5.5, 1)ellipse(14pt and 10pt);
            \draw[] (6.5, 1)ellipse(14pt and 10pt);
            \draw{
            (5.5,1)node[]{$G_1$}
            (6.5,1)node[]{$G_2$}
            (6,1)node[fill=black, minimum width = 2pt, label={[shift={(0,-.6)}]{$v$}}]{}
            };
        \end{tikzpicture}
    \caption{The graph $G=G_1\oplus_vG_2$ created from $G_1$ and $G_2$}
    \label{fig:vertexsum}
    \end{figure}

    \begin{lemma}[Theorem 2.1 of \cite{faught2021}]\label{thm:kem1sepformula}
        Let $G$ be a graph with a 1-separator $v$. Let $G_1, G_2$ be the two graphs of the
        1-separation so that $G=G_1\oplus_vG_2$ and let $m_1 = |E(G_1)|$ and $m_2 = |E(G_2)|$. Then
        we have
        \begin{align*}
            \kemeny(G)
            &= \frac{m_1\left(\kemeny(G_1) + \mu(G_2, v)\right) + m_2\left(\kemeny(G_2)
             + \mu(G_1, v)\right)}{m_1+m_2}.
        \end{align*}
    \end{lemma}

    Proposition 1.2 of \cite{breen2019computing} gives an expression for Kemeny's constant in terms of
    the coefficients of characteristic polynomial of the normalized Laplacian matrix. This result can
    be restated in terms of the transition probability matrix, which will be useful. This is also stated
    in more general form (i.e. for any regular Markov chain) in \cite{yilmaz2020kemeny}.
    
    \begin{lemma}\label{lem:KemenyFromCharPoly}
        Let $G$ be a connected graph, and let  $p(x)$ be the characteristic polynomial of the transition probability matrix for the
        random walk on $G$. Then if $p(1-x) = \cdots c_2x^2 + c_1x$ we have
        \begin{equation*}
            \kemeny(G) = -\frac{c_2}{c_1}.
        \end{equation*}
    \end{lemma}
        
    \subsection{Non-Backtracking Random Walks}
        Recently, there has been interest in \emph{non-backtracking} random walks on graphs
        \cite{alon2007non,kempton2016non,krzakala2013spectral,kempton2018polya,glover2021non,torres2021nonbacktracking};
        that is, a random walk on a graph where at each step you are not permitted to transition to the
        vertex you were at one step previously. Since the transition probabilities now depend not only
        on the current state of the system but also the previous state, a non-backtracking random walk
        on the vertex set of a graph will not be a Markov chain and as such, Kemeny's constant is not
        defined for such a walk. However, an equivalent walk can be defined on the directed edges of the
        graph, which produces a Markov chain in which we can account for the previous two states of the
        chain, but still has the Markov property, (see \cite{kempton2016non,glover2021non} for instance).

        Let $G$ be a graph with vertex set $V= \{1, 2, \ldots, n\}$ and edge set $E$. The oriented edge
        set of $G$ is $E'(G) = \{(i, j): \{i, j\} \in E\}$; each edge $\{u,v\}$ has been replaced by two
        directed arcs $(u, v)$ and $(v, u)$. An arc $(i,j)$ can also be written $i \to j$, and $i$ is
        referred to as the \emph{tail} of the arc, and $j$ is referred to as the \emph{head}. We define
        a \emph{random walk on the edge space of $G$} as a Markov chain whose states are the elements of
        $E'(G)$, with a positive transition probability $p_{(i, j), (k, l)}$ only if $j=k$. For the
        simple random walk on the edge space of $G$, if the current state is the arc $(i,j)$ the next
        edge is chosen at random from the edges incident with the head of that arc. In particular, the
        transition probabilities are:
        \begin{equation*}
            p_{(i,j), (k, l)} = \begin{cases}
                1 / \deg(j), &\text{if } k = j; \\
                0,           &\text{if } k \neq j.
            \end{cases}
        \end{equation*}

        The transition matrix for the random walk on the edge space of a graph can also be defined using
        matrices, described in the following definition.  See \cite{glover2021non} for
            a more in-depth study of these matrices.

        %BELOW: Definition containing all of the matrices
        \begin{definition}\label{def:nbmatrices}
            Let $G$ be a graph with vertex set $V$ and edge set $E$, and let $E'$ denote the oriented edge set of $G$.
            The \emph{startpoint incidence} operator of $G$ is the $n \times 2m$ matrix $T$ with
            rows indexed by $V$ and columns indexed by $E'$.
            \begin{align*}
                T(u, (v, w)) &= \begin{cases}
                    1, &\text{if } u = v; \\
                    0, &\text{otherwise.}
                \end{cases}%\\
                \end{align*}
            The \emph{endpoint incidence} operator of $G$  is the $2m \times n$ matrix $S$ with
            rows indexed by $E'$ and columns indexed by $V$.
                \begin{align*}
                S((u, v), w) &= \begin{cases}
                    1, &\text{if } v = w; \\
                    0, &\text{otherwise.}
                \end{cases}
                \end{align*}
            The \emph{edge reversal operator} $\tau$ is the $2m\times 2m$ matrix with rows and columns both indexed by $E'$ that switches a
            directed edge with its opposite.
                \begin{align*}
                \tau((u, v), (x, y)) &= \begin{cases}
                    1, &\text{if } v = x \text{ and } u = y; \\
                    0, &\text{otherwise}.
                \end{cases}
            \end{align*}
            The adjacency matrix of $G$ is $A = TS$, the \emph{edge adjacency matrix} is given by
            $C = ST$, and the \emph{non-backtracking edge adjacency matrix} is $B = ST - \tau$.
            Let $D_e$ be the diagonal degree matrix  where
            the diagonal entry corresponding to a directed edge $u \to v$ is $\deg(v)$. Then the \emph{edge
            space transition probability matrix} is $P_e = D_e^{-1}C$ and the \emph{non-backtracking
            transition probability matrix} is $P_{nb} = (D_e - I)^{-1}B.$ 
        \end{definition}
        %ABOVE: Definition containing all of the matrices

        We are now in a position to consider and define the value of Kemeny's constant for a
        non-backtracking random walk on a graph, and compare it with the value of Kemeny's constant for
        a simple random walk on the same graph. A key concern when comparing these random walks is that the state space is different for
        the two Markov chains under consideration. For this reason, we consider both
        as random walks on the edge space. The \emph{edge Kemeny's constant} of an undirected graph $G$,
        denoted $\kemeny_e(G)$, is the value of Kemeny's constant for the random walk on the directed
        edges of the graph, and the \emph{non-backtracking Kemeny's constant}, denoted $\kemeny_{nb}(G)$,
        is the value of Kemeny's constant for the Markov chain with transition matrix $P_{nb}$. To avoid
        ambiguity from this point onwards, we also denote by $\kemeny_v(G)$ the value of Kemeny's
        constant for the simple random walk on the vertices of $G$, and refer to it as the vertex Kemeny's constant.

        Our first main result relates the value of $\kemeny_e(G)$ with the value of $\kemeny_v(G)$ for
        any graph $G$. We first prove a technical lemma.

        \begin{lemma}\label{lem:DSisSD}
            Let $G$ be a graph with no isolated vertices, and let $S$ be the endpoint incidence operator, $D_e$ the degree matrix for the directed edges of the
            graph, and $D$ the degree matrix for the vertices. Then
            \begin{equation*}
                D_e^{-1}S = SD^{-1}.
            \end{equation*}
        \end{lemma}
        \begin{proof}
            A computation reveals that 
            \begin{equation*}
                (D_e^{-1}S)_{(u,v),j} = (SD^{-1})_{(u,v),j} = \begin{cases}
                        1 / \deg(v) &\text{if } j = v; \\
                        0           &\text{else.}
                    \end{cases}
            \end{equation*}
        \end{proof}

        \begin{theorem}\label{thm:kemSRWedge}
            Let $G$ be a connected graph with $|V(G)|=n$ and $|E(G)|=m$. Then
            \begin{equation*}
                \kemeny_e(G) = \kemeny_v(G) + 2m-n.
            \end{equation*}
        \end{theorem}
        \begin{proof}
            Let $S, T, A, D$, and $D_e$ be as in Definition \ref{def:nbmatrices}. 
            %the endpoint and startpoint incidence operators, with $TS=A$, the adjacency
            %matrix of $G$ and $ST=C$, the edge adjacency matrix for the edge space of $G$. Let $D_e$
            %be the degree matrix for the edge space simple random walk and $D$ the degree matrix for the
            %graph on the vertex space. 
            Recall that $A, D$ are symmetric matrices, and we denote by
            $X \sim Y$ the condition that $X$ and $Y$ are similar matrices.  Note that %By Lemma \ref{lem:DSisSD}
            \begin{align*}
                \begin{bmatrix}
                    P_e & 0 \\
                    T & 0
                \end{bmatrix} &=\begin{bmatrix}
                    D_e^{-1} ST & 0 \\
                    T & 0
                \end{bmatrix} \\
                &\sim \begin{bmatrix}
                    0 & 0 \\
                    T & T D_e^{-1} S
                \end{bmatrix} &\text{ (Theorem 1.3.22 of \cite{horn2012matrix})}\\
                &=\begin{bmatrix}
                    0 & 0 \\
                    T & TS D^{-1}
                \end{bmatrix}&\text{(Lemma \ref{lem:DSisSD})}\\
                &=\begin{bmatrix}
                    0 & 0 \\
                    T &  AD^{-1}
                \end{bmatrix} \\
                &\sim\begin{bmatrix}
                    0 & 0 \\
                    T & D^{-1} A
                \end{bmatrix}\\
                &=\begin{bmatrix}
                    0 & 0 \\
                    T & P
                \end{bmatrix}.
            \end{align*}
            Therefore the eigenvalues of $P_e$ are the eigenvalues of $P$ with an additional $2m-n$ zero
            eigenvalues. %Let the first $n$ eigenvalues of $P_e$ be the eigenvalues shared with $P$ and $\rho_1=1$. 
            Suppose that the eigenvalues of $P_e$ are ordered so that $\rho_1(P_e) = 1$ and the first $n$ eigenvalues are those shared with $P$. 
            It then follows that
            \begin{equation*}
                \kemeny_e(G)
                = \sum_{i=2}^{2m}\frac{1}{1-\rho_i(P_e)}
                = \sum_{i=2}^n\frac{1}{1-\rho_i(P)} + \sum_{i=n+1}^{2m}\frac{1}{1-0}
                = \kemeny_v(G) + 2m-n.
            \end{equation*}
        \end{proof}

        The remainder of the work in this paper explores the relationship between  the vertex, edge,
        and non-backtracking variants of Kemeny's constant by analysing and deriving relationships
        between these for certain families of graphs. In Section 3, we compare 
        $\kemeny_e(G)$ and $\kemeny_{nb}(G)$ for regular graphs by considering the difference and ratio of these quantities, and in Section 4 we explore the same
        for biregular graphs. In Section 5 we derive exact results for cycle barbell graphs, a new
        family we define to better outline and explore the differences between the non-backtracking and
        vertex Kemeny's constant, in order to develop our intuition around the behavior of Kemeny's
        constant and quantifying how it is affected by imposing non-backtracking conditions on a random
        walk.

\section{Regular Graphs}
    % \begin{definition}
    %     A graph $G$ is \emph{d-regular} if each vertex of $G$ has $d$ neighbors.
    % \end{definition}

    In this section we consider connected $d$-regular graphs with $d \geq 3$. 
    We do not consider regular graphs
    of lower degree as in those cases the non-backtracking Kemeny's constant will not be well defined (for $d=2$, the matrix $P_{nb}$ is reducible). %\red{Prop 2.3 Cory's Paper says NBRW is irreducible if $d_{\text{min}}\geq2$ and not a cycle}
    
    While Theorem \ref{thm:kemSRWedge} already gives a nice expression for the edge Kemeny's constant,
    we now write it in terms of the adjacency eigenvalues in the case of regular graphs, which will be
    useful to make the comparison between the edge Kemeny's constant and the non-backtracking Kemeny's
    constant.

    \begin{lemma}\label{lem:dregsimple}
        Let $G$ be a connected $d$-regular graph of order $n$, where $d\geq 3$, with adjacency spectrum $d=\lambda_1 >
        \lambda_2 \geq \cdots \geq \lambda_n$. Then
        \begin{align*}
            \kemeny_e(G) &= n(d-1) + \sum_{i=2}^n\frac{d}{d-\lambda_i}.
        \end{align*}
    \end{lemma}
    \begin{proof}
        For a $d$-regular graph on $n$ vertices, $2|E(G)| - |V(G)| = n(d-1)$ and Theorem
        \ref{thm:kemSRWedge} asserts that $\kemeny_e(G) = n(d-1) + \kemeny_v(G)$. Due to the
        regularity of the graph, the transition matrix $P = D^{-1} A = \frac{1}{d}A$ has eigenvalues
        $\rho_i = \lambda_i/d$ and
        \begin{equation*}
            \kemeny_v(G)
            = \sum_{i=2}^n\frac{1}{1-\frac{\lambda_i}{d}}
            = \sum_{i=2}^n\frac{d}{d-\lambda_i}
        \end{equation*}
        as per Lemma \ref{lem:kemeny_def}. The statement follows directly.
        
        %Let $d = \lambda_1 > \lambda_2 \geq \hdots \geq \lambda_n $ be the eigenvalues of the adjacency matrix of $G$. It is known that the eigenvalues of the edge space adjacency matrix $C$ are the eigenvalues of $A$ and an  additional $2m-n$ zero eigenvalues. Since $G$ is $d$-regular, the eigenvalues of the transition matrix for the simple random walk on the edge space are $\lambda_i/d$ for all $\lambda_i$ and still $2m-n$ zeros. Then Kemeny's constant is given by
        %\begin{align*}
        %    \kemeny_e(G)
        %    &= \frac{2m-n}{1-0} + \sum_{i=2}^n \frac{1}{1-\frac{\lambda_i}{d}} \\
        %    &= 2 \left(\frac{dn}{2}\right)-n + \sum_{i=2}^n \frac{d}{d-\lambda_i} \\
        %    &= n(d-1) + \sum_{i=2}^n \frac{d}{d -\lambda_i}
        %\end{align*}
    \end{proof}

    %It is worth noting that Theorem \ref{thm:dregsimple} could also be written as $\kemeny_e(G) = n(d-1) + \kemeny_v(G)$ where $\kemeny_v(G)$ denotes Kemeny's constant for a simple random walk on the vertex space. 

    \begin{theorem}\label{thm:dregnb}
        Let $G$ be a connected $d$-regular graph of order $n$, where $d\geq 3$. Then
        \begin{align*}
            \kemeny_{nb}(G) &= \frac{(d-2)\kemeny_e(G)}{d}+2n+\frac{1}{d-2}-\frac{n}{d}.
        \end{align*}
    \end{theorem}
    \begin{proof}
        Theorem 5 of \cite{kempton2016non} states that the spectrum of the non-backtracking transition
        probability matrix of a $d$-regular graph is
        \begin{equation*}
                \left\{
                  \left(
                        \frac{1}{d-1}
                    \right)^{m-n}, \left(
                        \frac{-1}{d-1}
                    \right)^{m-n}, 
                        \frac{\lambda_i \pm \sqrt{\lambda_i^2-4(d-1)}}{2(d-1)}
                \right\},
        \end{equation*}
        where $\lambda_i$ ranges over the eigenvalues of the adjacency matrix.  Noting that for $\lambda_1=d$, this yields the eigenvalues of $1$ and $\frac{1}{d-1}$, using Lemma \ref{lem:kemeny_def} we obtain:
        \begin{align*}
            \kemeny_{nb}(G)
            &= \frac{m-n+1}{1-\frac{1}{d-1}} + \frac{m-n}{1+\frac{1}{d-1}}
             + \sum_{i=2}^n \left[\frac{1}{1-\frac{\lambda_i+\sqrt{\lambda_i^2-4(d-1)}}{2(d-1)}}+ \frac{1}{1-\frac{\lambda_i-\sqrt{\lambda_i^2-4(d-1)}}{2(d-1)}}\right] \\
           % &= \frac{\frac{dn}{2}-n+1}{\frac{d-2}{d-1}}+\frac{\frac{dn}{2}-n}{\frac{d}{d-1}}
            % + \sum_{i=2}^n \frac{2(d-1)-\lambda_i}{d-\lambda_i} \\
            &= (d-1)\left[\frac{nd(d-2)+2d+n(d-2)^2}{2d(d-2)}\right]
             + \sum_{i=2}^n\frac{d-2+d-\lambda_i}{d-\lambda_i} \\
            &= (d-1)\left[\frac{2nd^2-6nd+2d+4n}{2d(d-2)}\right]
             + \sum_{i=2}^n\left(1+\frac{d-2}{d-\lambda_i}\right) \\
            &= (d-1)\left[\frac{n(d-2)(d-1)+d}{d(d-2)}\right]
             + (n-1)+(d-2)\sum_{i=2}^n\frac{1}{d-\lambda_i} \\
            &= (d-1)\left[\frac{n(d-1)}{d}+\frac{1}{d-2}\right]
             + n-1+\frac{(d-2)}{d}\left(\kemeny_e(G)-n(d-1)\right) \\
            %&= \frac{n(d-1)^2}{d}+\frac{d-1}{d-2} +n-1+\frac{(d-2)}{d}\left(\kemeny_e(G)-n(d-1)\right) \\
            %&= \frac{n(d-1)}{d}+\frac{(d-2)\kemeny_e(G)}{d}+\frac{d-1}{d-2}+n-1 \\
            %&= \frac{n(d-1)+d(n-1)}{d}+\frac{(d-2)\kemeny_e(G)}{d}+\frac{d-1}{d-2} \\
            %&= \frac{2nd-n-d}{d}+\frac{(d-2)\kemeny_e(G)}{d}+\frac{d-1}{d-2} \\
            %&= 2n-1 -\frac{n}{d}+\frac{(d-2)\kemeny_e(G)}{d}+\frac{d-1}{d-2} \\
            &= 2n -\frac{n}{d}+\frac{(d-2)\kemeny_e(G)}{d}+\frac{1}{d-2}.
        \end{align*}
    \end{proof}

    Now that we have these expressions it is natural to compare them by looking at both the difference
    and ratio of them.

    \begin{theorem}\label{thm:regdiff}
        Let $G$ be a $d$-regular graph, $d\geq 3$, which is not $K_4, K_5$, or $K_{3,3}$. Then
        \[
        \kemeny_e(G)>\kemeny_{nb}(G).
        \]
        %For all $d$-regular graphs except three, Kemeny's constant of a simple random walk on the edge
        %space is larger than Kemeny's constant of a non-backtracking random walk on the edge space.
    \end{theorem}
    \begin{proof}
        From Lemma \ref{lem:dregsimple} and Theorem \ref{thm:dregnb} we have the following.
        \begin{align*}
            \kemeny_e(G) - \kemeny_{nb}(G)
            &= \kemeny_e(G) - \left(
                \frac{(d-2) \kemeny_e(G)}{d} + 2n + \frac{1}{d-2} - \frac{n}{d}
            \right) \\
            &= \frac{n}{d} - 2n - \frac{1}{d-2} + \kemeny_e(G) \left(1 - \frac{d-2}{d}\right) \\
            &= \frac{2\kemeny_e(G)}{d}-\frac{1}{d-2}-n\left(2-\frac{1}{d}\right) \\
            &= \frac{2\left(n(d-1)+\kemeny_v(G)\right)}{d}-\frac{1}{d-2}-n\left(2-\frac{1}{d}\right).
        \end{align*}
        As $d\geq3$, to see when $\kemeny_e(G) - \kemeny_{nb}(G)\geq0$, we can consider when $d(d-2)(\kemeny_e(G) -\kemeny_{nb}(G))\geq0$. %Doing this we get the following.
        % \begin{align*}
        %     0&\leq d(d-2)\left(\kemeny_e(G) - \kemeny_{nb}(G)\right)\\
        %     &=2n(d-2)(d-1)+2(d-2)\kemeny_v(G)-d-n(d-2)(2d-1)\\
        %     &= 2(d-2)\kemeny_v(G)-d-n(d-2)
        % \end{align*}
        Simplifying this expression, we have
        \begin{align*}
            d(d-2)\left(\kemeny_e(G) - \kemeny_{nb}(G)\right) & =2n(d-2)(d-1)+2(d-2)\kemeny_v(G)-d-n(d-2)(2d-1)\\
            & = 2(d-2)\kemeny_v(G)-d-n(d-2).
        \end{align*}
        It is known that $\kemeny_v(G) \geq \kemeny_v(K_n) = \frac{(n-1)^2}{n}$. Substituting this into the above expression gives the following inequality:
        % \begin{align*}
        %     0&\leq \frac{2(d-2)(n-1)^2}{n}-d-n(d-2)\\
        %     0&\leq 2(d-2)(n-1)^2-nd-n^2(d-2)\\
        %     &= d(n^2-5n+2) - 2(n^2-4n+2)\\
        %     &= n^2(d-2)+n(8-5d)+2d-4
        % \end{align*}
        \begin{eqnarray*}
            d(d-2)\left(\kemeny_e(G) - \kemeny_{nb}(G)\right)&\geq& \frac{2(d-2)(n-1)^2}{n}-d-n(d-2)\\
            & = & \tfrac{1}{n}(2(d-2)(n-1)^2-nd-n^2(d-2))\\
            & = & \tfrac{1}{n}(d(n^2-5n+2) - 2(n^2-4n+2))\\
            & = & \tfrac{1}{n}(n^2(d-2)+n(8-5d)+2d-4).
        \end{eqnarray*}
        Analyzing the expression we see that it is positive for $d=3$ when $ n\geq 7$ and for $d=4,5,6$
        when $n\geq 6$. For $d>6$ the expression will be positive whenever $n \geq 5$. Thus
        $\kemeny_e(G) > \kemeny_{nb}(G)$ except for potentially a finite number of graphs. Checking
        these, we find that there are only three regular graphs for which $\kemeny_{nb}(G) \geq
        \kemeny_e(G)$, namely $K_4, K_5,$ and $K_{3,3}$ with equality holding for $K_{3,3}$.
    \end{proof}

    %K_{2,2,2}: K_e = 22+1/3,  KNB = 22+1/6
    %K_{3,3}: K_3 = KNB = 16.5
    
    %Show the complete graph formulas as an example, complete bipartite
    
    %Get expression for ratio K/Knb 
    %Examples: K_n and K_{n,n} ratio converges to 1 very quickly. For necklaces, converges to 3
    \begin{theorem}\label{thm:kemratio}
        Let $G$ be a $d$-regular graph, $d\geq 3$, which is not $K_4$, $K_5$ or $K_{3, 3}$. Then \[1-\frac{2}{d} \; < \; \frac{\kemeny_{nb}(G)}{\kemeny_e(G)} \;< \; 1.\]
    \end{theorem}
    %Not confident on getting a sharper lower bound but maybe want it written different?
    \begin{proof}
        %(USE: Kemeny vertex as small as $n$ and as big as $n^3$ essentially)
        Using Lemma \ref{lem:dregsimple} and Theorem \ref{thm:dregnb} we see that
        \begin{eqnarray*}
            \frac{\kemeny_{nb}(G)}{\kemeny_e(G)}
            &=& \frac{d-2}{d} + \frac{2n}{\kemeny_e(G)} + \frac{1}{(d-2)\kemeny_e(G)}
             - \frac{n}{d\kemeny_e(G)} \\
            &=& 1 - \frac 2d + \frac{2n}{\kemeny_e(G)} + \frac{1}{(d-2)\kemeny_e(G)} 
             - \frac{n}{d\kemeny_e(G)}.
        \end{eqnarray*}
        The lower bound is easy to see since
        $\frac{2n}{\kemeny_e(G)}+\frac{1}{(d-2)\kemeny_e(G)}-\frac{n}{d\kemeny_e(G)} > 0$ for all
        values $n, d$ consistent with a regular graph. The upper bound follows from Theorem \ref{thm:regdiff}.
    \end{proof}
    After obtaining these bounds, one might ask how good they are. To investigate this question, we consider graph families which are extremal or conjectured to be extremal in some way.
    \begin{example}
        Complete graphs are known to have the smallest vertex Kemeny's constant among graphs of order $n$. Using the above results it is seen that $\lim_{n \to \infty} \kemeny_e(G)
    - \kemeny_{nb}(G) = 1$, and consequently $\lim_{n\to \infty}\frac{\kemeny_{nb}(G)}{\kemeny_e(G)} = 1$.
    \end{example}
    
    \begin{example}
        Necklace graphs are families of 3-regular graphs known to have large vertex Kemeny's constant.  Indeed, it is conjectured in \cite[Open Problem 6.14]{aldous-fill-2014} that the necklace graph on $n$ vertices is the extremal graph among all regular graphs that maximizes the value of $\kemeny_v(G)$.
        
        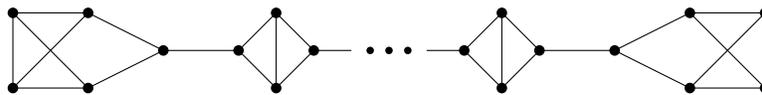
\begin{figure}[H]
        \centering
        \begin{tikzpicture}
        \tikzstyle{every node}=[circle,  fill=black, minimum width=4pt, inner sep=1pt]
            \draw{
            %first end
            (0,0)node{}--(0,1)node{}--(1,1)node{}--(2,.5)node{}--(1,0)node{}--(0,0)--(1,1)
            (0,1)--(1,0)
            %First Bead
            (2,.5)--(3,.5)node{}--(3.5,1)node{}--(4,.5)node{}--(3.5,0)node{}--(3.5,1)
            (3,.5)--(3.5,0)
            (4,.5)--(4.5,.5) %half edge
            %Dot dot dot
            (4.75,0.5)node[minimum size = 1pt]{}
            (5,0.5)node[minimum size = 1pt]{}
            (5.25,0.5)node[minimum size = 1pt]{}
            %Second Bead
            (5.5,.5)--(6,.5)node{}--(6.5,1)node{}--(7,.5)node{}--(6.5,0)node{}--(6.5,1)
            (6,.5)--(6.5,0)
            %Other end
            (7,.5)--(8,.5)node{}--(9,1)node{}--(10,1)node{}--(10,0)node{}--(9,0)node{}--(8,.5)
            (9,0)--(10,1)
            (10,0)--(9,1)
            };
            \end{tikzpicture}
            \caption{A necklace graph, as in the figure above, is a 3-regular graph on $n = 4k+2$ vertices where there are $k$ subgraphs (referred to as \emph{beads}) linked in a line, and the two subgraphs on the end are as shown, distinct from the $k-2$ in the middle. }
            \label{fig:necklacegraph}
        \end{figure}
        
        As these graphs afford many 1-separations, we can use the methods of \cite{faught2021} to obtain an explicit formula for the vertex Kemeny's constant. If $G$ is a necklace graph on $n$ vertices, these techniques give
        \[
        \kemeny_v(G) = \frac{4n^3+3n^2-122n+216}{16n}.
        \]
        Then applying the previous results gives 
        \begin{align*}
            \kemeny_e(G) &= \frac{4n^3+35n^2-122n+216}{16n} & \kemeny_{nb}(G) &=\frac{4n^3+115n^2-74n+216}{48n}.
        \end{align*}
        From these expressions it is then readily seen that $\lim_{n\to\infty}\frac{\kemeny_{nb}(G)}{\kemeny_e(G)} = 1/3$. That is, the family of necklace graphs achieves the lower bound on the ratio given in Theorem \ref{thm:kemratio} in the limit. 
    \end{example}

% The subsection that was Necklace graphs has been moved to "OtherThings"
        
We remark that, while the upper bound in Theorem \ref{thm:kemratio} is approached by complete graphs, if we fix the degree, we can prove a stronger upper bound.
        
   \begin{theorem}\label{thm:fix-d-upper-ratio}
        For a family $\{G_k\}$ of $d$-regular graphs with $d$ fixed, $d\geq 3$, and $|V(G_k)|\rightarrow\infty$ as $k\rightarrow\infty$, we have
        \begin{equation*}
            \lim_{k\rightarrow\infty}\frac{\kemeny_{nb}(G_k)}{\kemeny_e(G_k)} \leq 1-\frac{1}{d^2}.
        \end{equation*}
    \end{theorem}
    \begin{proof}
        We bound the expression in the proof of Theorem \ref{thm:kemratio}.  Let $G_k$ be a graph from the family with $n$ vertices and $m$ edges. By Theorem \ref{thm:kemSRWedge},
        \begin{equation*}
            \kemeny_e(G_k) =  \kemeny_v(G_k) + 2m-n = \kemeny_v(G_k) + n(d-1).
        \end{equation*}
        Since the complete graph has the smallest vertex Kemeny's constant for any graph on $n$ vertices
        we further get that
        \begin{eqnarray*}
            \kemeny_e(G_k) &\geq &\kemeny_v(K_n) + n(d-1)\\
            &=& \frac{(n-1)^2}{n} + n(d-1)\\
            &=& n-2+\frac{1}{n}+n(d-1)\\
            &\geq & nd-2.
        \end{eqnarray*}
        Then using this bound on $\kemeny_e(G_k)$ and taking the limit as $n\to\infty$ the ratio is bounded
        as follows.
        \begin{eqnarray*}
            \lim_{n\to\infty}\frac{\kemeny_{nb}(G_k)}{\kemeny_e(G_k)}
            &\leq &\lim_{n\to\infty} \left(
                1 - \frac 2d + \frac{2n}{nd-2} + \frac{1}{(d-2)(nd-2)} - \frac{n}{d(nd-2)}
            \right) \\
            &= & 1 - \frac{1}{d^2}.
        \end{eqnarray*}
    \end{proof}

%\textcolor{red}{Include remark regarding Ramanujan graphs asymptotically have $\frac{\kemeny_{nb}}{\kemeny_e}\geq 1-\frac{1}{d^2}-O(\frac{1}{d^{5/2}})$}
\begin{example}
    Given any $d$-regular Ramanujan graph, the ratio of the non-backtracking Kemeny's constant to the edge Kemeny's constant will be close to the upper bound in Theorem \ref{thm:fix-d-upper-ratio}. Recall that a graph is a \emph{Ramanujan graph} if its adjacency eigenvalues have $\lambda_2,|\lambda_n| \leq 2\sqrt{d-1}$. Using this bound on the eigenvalues and Lemma \ref{lem:dregsimple}, we can show 
    \[
    \kemeny_e(G) \leq n\left(d-1 + \frac{d}{d-2\sqrt{d-1}}\right).
    \]
    Then using Theorem \ref{thm:kemratio} leads to the bound
    \[
    \frac{\kemeny_{nb}(G)}{\kemeny_e(G)} \geq 1 - \frac{1}{d^2}\left(\frac{d+2\sqrt{d-1}}{d-2\sqrt{d-1}+\frac{2\sqrt{d-1}}{d}}\right) = 1 - \frac{1}{d^2} - O\left(\frac{1}{d^{5/2}}\right).
    \]
\end{example}

\section{Biregular Graphs}

In this section, we extend results about regular graphs to the case of biregular graphs.  A \emph{$(c,d)$-biregular graph} is a bipartite graph in which every vertex in one part of the bipartition has degree $c$ and every vertex in the other part has degree $d$. See Figure \ref{fig:bireg} for an example.
\begin{figure}[h!]
    \centering
    \begin{tikzpicture}
    \tikzstyle{every node}=[circle,  fill=black, minimum width=4pt, inner sep=1pt]
    \draw (-1,0)node{}--(-1.5,1)node[draw,fill=white]{}--(-2,0)node{}--(-1.5,-1)node[draw,fill=white]{}--(-1,0) (1,0)node{}--(1.5,1)node[draw,fill=white]{}--(2,0)node{}--(1.5,-1)node[draw,fill=white]{}--(1,0) (-1.5,1)--(0,1)node{}--(1.5,1)  (-1.5,-1)--(0,-1)node{}--(1.5,-1);
    \end{tikzpicture}
    \caption{A $(2,3)$-biregular graph on $10$ vertices}
    \label{fig:bireg}
\end{figure}
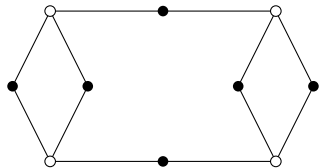

%\textcolor{red}{Example?}

    \begin{lemma}\label{lem:biregsimple}
        Kemeny's constant for a simple random walk on the edge space of a $(c,d)$-biregular graph is given by
        \begin{equation*}
            \kemeny_e(G)=2m-n+\sum_{i=2}^n\frac{\sqrt{cd}}{\sqrt{cd}-\lambda_i}.
        \end{equation*}
%        If $G$ is biregular complete, then
 %       \begin{equation*}
 %           \kemeny_e(G)=2cd-c-d+\sum_{i=2}^n\frac{\sqrt{cd}}{\sqrt{cd}-\lambda_i} = 2cd - \frac{3}{2}.
 %       \end{equation*}
    \end{lemma}

    \begin{proof}
    %\red{Probably shorten this proof slightly as we did with regular graphs}
    As in the regular case, Theorem \ref{thm:kemSRWedge} already gives a nice expression for $\kemeny_e(G)$. We give this expression in terms of the adjacency eigenvalues of $G$ to assist the comparison with $\kemeny_{nb}(G)$ later.
    
    It is known that for a bipartite biregular graph, the eigenvalues of the transition probability matrix are of the form $\lambda_i/\sqrt{cd}$ where $\lambda_i$ is an eigenvalue of the adjacency matrix (see proof of Corollary 2 in \cite{kempton2016non}). Using Lemma \ref{lem:kemeny_def} gives the result.
    
        %Let $\sqrt{cd}>\lambda_2\geq\dotsb\geq\lambda_n$ be the eigenvalues of the adjacency matrix $A$. From \cite{kempton2016non} Corollary 2 \red{I think},
        %we know that the eigenvalues of $P$ are $\frac{\lambda_i}{\sqrt{cd}}$. We also know that the eigenvalues of
        %the edge space adjacency matrix $C$ are the eigenvalues of $A$ and an addition $2m-n$ zero eigenvalues.
        %Thus the eigenvalues of the edge space transition matrix are $2m-n$ zeroes and the eigenvalues of $P$.
        %Thus Kemeny's constant is
        %\begin{equation*}
        %    \kemeny_e(G)
        %    = 2m - n + \sum_{i=2}^n \frac{1}{1-\frac{\lambda_i}{\sqrt{cd}}}
        %    = 2m - n + \sum_{i=2}^n \frac{\sqrt{cd}}{\sqrt{cd}-\lambda_i}.
        %\end{equation*}
    %    The second results follows from a biregular complete graph have $cd$ edges and $c+d$ nodes and $K_v(G) = n-3/2$.
    \end{proof}
    
    \begin{theorem}\label{thm:biregnbkem}
        Let $G$ be a $(c,d)-$biregular graph, $r$ be the number of vertices with degree $c$, and $s$
        the number of vertices with degree $d$. Without loss of generality suppose that $r\geq s$. Then 
        %\begin{align*}
        %    \kemeny_{nb}(G) =& \frac{2(m-n+1)(c-1)(d-1)}{(c-1)(d-1) - 1} + \frac{2(r-s)(d-1)}{d}+\frac12\\
        %    &+ 2\sum_{i=2}^s\frac{(\lambda_i - \sqrt{2cd-c-d})(\lambda_i + \sqrt{2cd-c-d})}{(\lambda_i-\sqrt{cd})(\lambda_i+\sqrt{cd})}.
        %\end{align*}
        %\red{REWRITTEN:}
        \begin{align*}
            \kemeny_{nb}(G) &= \frac{2(m-n+1)(c-1)(d-1)}{(c-1)(d-1)-1} + \frac{2(r-s)(d-1)}{d}+\frac{1}{2}\\
            &\qquad \qquad+ 2(s-1) + \frac{cd-c-d}{cd}\left[\kemeny_e(G) - 2m + n-\frac{1}{2}-(r-s)\right].
        \end{align*}
    \end{theorem}
    \begin{proof}
From \cite{kempton2016non} we know that the eigenvalues of the non-backtracking transition probability matrix for a $(c,d)$-biregular graph $G$ are as follows:
\[\{(\pm\alpha)^{m-n}, (\pm i\alpha)^{r-s}, \pm \sqrt{A_k\pm B_k}\},\]
where
\begin{eqnarray*}
\alpha & = & \frac{1}{\sqrt{(c-1)(d-1)}};\\
A_k & = & \frac{\lambda_k^2 - (c-1) - (d-1)}{2(c-1)(d-1)};\\
B_k & = & \frac{(\lambda_k^2-(c-1)-(d-1))^2-4(c-1)(d-1)}{4(c-1)^2(d-1)^2};
\end{eqnarray*}
for $k=1, \ldots, s$, and thus $\lambda_k$ ranges over the $s$ largest eigenvalues of the adjacency matrix of $G$.

We calculate $\kemeny_{nb}(G)$ using the formula in Lemma \ref{lem:kemeny_def}. To this end, some simple computation and simplification shows that
\[(m-n)\left(\frac{1}{1-\alpha} + \frac{1}{1+\alpha}\right) = \frac{2(m-n)(c-1)(d-1)}{(c-1)(d-1)-1}\]
and
\[(r-s)\left(\frac{1}{1-i\alpha} + \frac{1}{1+i\alpha}\right) = \frac{2(r-s)(d-1)}{d}.\]
Finally, for fixed $k$ we compute
\[\frac{1}{1-\sqrt{A_k+B_k}}+ \frac{1}{1-\sqrt{A_k+B_k}}+\frac{1}{1+\sqrt{A_k+B_k}}+\frac{1}{1+\sqrt{A_k-B_k}} = \frac{4(1-A_k)}{(1-A_k)^2-B_k}.\]
Some tedious simplification gives the expression
\[2\cdot\frac{(\lambda_k - \sqrt{2cd-c-d})(\lambda_k+\sqrt{2cd-c-d})}{(\lambda_k-\sqrt{cd})(\lambda_k+\sqrt{cd})}.\]

 In this form it is clear that we must look at the four eigenvalues that come from the adjacency matrix eigenvalue $\sqrt{cd}$ separately.
        A straightforward computation reveals that the eigenvalue $\sqrt{cd}$ of the adjacency matrix will give rise to the following eigenvalues for the non-backtracking transition probability matrix:
        \[
        1, -1, \frac{1}{\sqrt{(c-1)(d-1)}}, -\frac{1}{\sqrt{(c-1)(d-1)}}.
        \]
        Combining everything we have currently gives
        \begin{align*}
            \kemeny_{nb}(G) &= \frac{2(m-n+1)(c-1)(d-1)}{(c-1)(d-1) - 1} + \frac{2(r-s)(d-1)}{d}+\frac12\\
            & \qquad \qquad + 2\sum_{i=2}^s\frac{(\lambda_i - \sqrt{2cd-c-d})(\lambda_i + \sqrt{2cd-c-d})}{(\lambda_i-\sqrt{cd})(\lambda_i+\sqrt{cd})}.
        \end{align*}
        However, further work will allow for easier comparison to $\kemeny_e(G)$. We begin by rearranging the summation term.
        \begin{align*}
        \frac{(\lambda_k - \sqrt{2cd-c-d})(\lambda_k+\sqrt{2cd-c-d})}{(\lambda_k-\sqrt{cd})(\lambda_k+\sqrt{cd})} &=1 + \frac{cd-c-d}{(\sqrt{cd}-\lambda_k)(\sqrt{cd}+\lambda_k)}\\
        &= 1 + (cd-c-d)\left[\frac{1}{2\sqrt{cd}(\sqrt{cd}-\lambda_k)} +\frac{1}{2\sqrt{cd}(\sqrt{cd}+\lambda_k)} \right]
        \end{align*}
        %Recall that \red{Do we need this recall here? Or could we just reference Theorem \ref{thm:biregsimple}}
        %\[
        %\kemeny_e(G) = 2m - n + \sum_{i=2}^n\frac{\sqrt{cd}}{\sqrt{cd}-\lambda_i}
        %\]
Then, using Lemma \ref{lem:biregsimple} and the fact that the adjacency spectrum is symmetric about 0 with null space at least dimension $r-s$, we get
\begin{align*}
            &2\sum_{k=2}^s\frac{(\lambda_k - \sqrt{2cd-c-d})(\lambda_k + \sqrt{2cd-c-d})}{(\lambda_k-\sqrt{cd})(\lambda_k+\sqrt{cd})}\\
            &= 2\sum_{k=2}^s\left(1 + \frac{cd-c-d}{(\sqrt{cd}-\lambda_k)(\sqrt{cd}+\lambda_k)}\right)\\
            &= 2(s-1) + 2(cd-c-d)\sum_{k=2}^s\frac{1}{(\sqrt{cd}+\lambda_k)(\sqrt{cd}-\lambda_k)}\\
            &= 2(s-1) + \frac{cd-c-d}{\sqrt{cd}}\sum_{k=2}^s\left(\frac{1}{\sqrt{cd}+\lambda_k} + \frac{1}{\sqrt{cd}-\lambda_k}\right)\\
            &= 2(s-1) + \frac{cd-c-d}{\sqrt{cd}}\left[\sum_{k=2}^n\left(\frac{1}{\sqrt{cd}-\lambda_k}\right)- \left(\frac{1}{2\sqrt{cd}} + \frac{(r-s)}{\sqrt{cd}}\right)\right]\\
            &= 2(s-1) + \frac{cd-c-d}{\sqrt{cd}}\left(\frac{\kemeny_e(G) - 2m + n}{\sqrt{cd}}\right)-\frac{cd-c-d}{\sqrt{cd}}\left(\frac{1}{2\sqrt{cd}} + \frac{r-s}{\sqrt{cd}} \right)\\
            &= 2(s-1) + \frac{cd-c-d}{cd}\left[\kemeny_e(G)\right]+\frac{cd-c-d}{cd}\left[-2m+n-\frac{1}{2}-(r-s)\right].
        \end{align*}
        \end{proof}

    In order to bound $\kemeny_e(G) - \kemeny_{nb}(G)$ it will be useful to have the following bound.
    \begin{lemma}\label{lem:kemebipartitebound}
    If $G$ is bipartite, then
    \[
    \kemeny_e(G) \geq 2m - \frac{3}{2}
    \]
    with equality if and only if $G$ is complete bipartite.
    \end{lemma}
    \begin{proof}
    In \cite[Prop 4.1]{ciardo2020kemeny} it was shown that $\kemeny_v(G) \geq n - \frac{3}{2}$ with equality if and only if $G$ is complete bipartite. Combining this with Theorem \ref{thm:kemSRWedge} gives the result.
    \end{proof}
    
    \begin{theorem}\label{thm:BipartiteDifference}
    Let $G$ be a $(c,d)$-biregular graph which is not $K_{2,3}, K_{2,4}, K_{2,5},$ or $K_{3,3}$. Then $\kemeny_e(G) >\kemeny_{nb}(G).$
    
    %If $c=2$ and $d\geq6$ or if $c,d\geq3$ then $\kemeny_e(G) - \kemeny_{nb}(G) >\geq 0.$
    \end{theorem}
    \begin{proof}
    From Theorem \ref{thm:biregnbkem} one can compute
    \begin{eqnarray*}
        \kemeny_e(G) - \kemeny_{nb}(G) &=&\: \kemeny_e(G)\left(\frac{c+d}{cd}\right) - \frac{2(m-n+1)(c-1)(d-1)}{(c-1)(d-1)-1} - \frac{1}{2} - 2(s-1)\\
        & & \qquad + (r-s)\left[\frac{-2(d-1)}{d} + \frac{cd-c-d}{cd}\right] + \frac{(2m - n + \frac{1}{2})(cd-c-d)}{cd}\\
       % =&\kemeny_e(G)\left(\frac{c+d}{cd}\right) - 2(m-n+1)\left(1 + \frac{1}{(c-1)(d-1)-1}\right) - 2s +\frac{3}{2}\\
        %&+\left(1 - \frac{c+d}{cd}\right)(2m-n+\frac{1}{2}+r-s) - \frac{2(r-s)(d-1)}{d}\\
        & = & \kemeny_e(G)\left(\frac{c+d}{cd}\right) - 2(m-n+1)\left(1 + \frac{1}{(c-1)(d-1)-1}\right) + \frac{3}{2} - 2r + \frac{2(r-s)}{d}\\
        & & \qquad +\left(1 - \frac{c+d}{cd}\right)(2sd - 2s + \frac{1}{2})\qquad \qquad \text{ (since $m= sd$ and $n = r+s$)}\\
        & \geq& (2m-\frac{3}{2})\left(\frac{c+d}{cd}\right) + \left(1 - \frac{c+d}{cd}\right)(2s(d-1) + \frac{1}{2})\\
        & & \qquad - 2(m-n+1)\left(1 + \frac{1}{(c-1)(d-1)-1}\right) + \frac{2(r-s)}{d} + \frac{3}{2} - 2r\text{ (by Lemma \ref{lem:kemebipartitebound})}\\
        %=&\:2m\left(\frac{c+d}{cd}\right) - \frac{3}{2}\frac{c+d}{cd} + 2m - 2s + \frac{1}{2} - 2m\left(\frac{c+d}{cd}\right) + 2s\left(\frac{c+d}{cd}\right) - \frac{c+d}{2cd}\\
        %&-2m + 2n - 2 - \frac{2(m-n+1)}{(c-1)(d-1)-1} + \frac{2(r-s)}{d} + \frac{3}{2} - 2r \text{ (since $m = sd$)}\\
        %=&\:\frac{-2(c+d)}{cd} - 2s + \frac{2s(c+d)}{cd} + 2n + \frac{2(r-s)}{d} - 2r - \frac{2(m-n+1)}{(c-1)(d-1)-1}\\
        %=&\:\frac{2(s-1)(c+d)}{cd} + \frac{2(r-s)}{d} - \frac{2(m-n+1)}{(c-1)(d-1)-1} \text{ (since $r + s = n$)}\\
        %=&\:-\frac{2}{c} - \frac{2}{d} - \frac{2}{cd-c-d} - \frac{2m}{cd-c-d} + \frac{2n}{cd-c-d} + \frac{2r}{d} + \frac{2s}{c}\\
        & = & 2\left[\frac{n-m-1}{cd-c-d} + \frac{r}{d} + \frac{s}{c}\right] - \left(\frac{2}{c}+\frac{2}{d}\right).% \text{ (since $c\geq 2, d\geq 3$)}
    \end{eqnarray*}
    
    %(Second to last line is from Mathematica expand. Also, important to note that $cd-c-d = (c-1)(d-1)-1$. I think we switched at the end since I used Mathematica but we should probably try to keep that consistent.)
    
    The difference between $m,n$ will be greatest when $G = K_{c,d}$; that is, when $m = cd$ and $n = c+d$. In that case, $n-m-1 = -(c-1)(d-1)$. Also note that $r\geq d$ and $s\geq c$. Then the expression above is greater than or equal to

    \begin{align*}
        2\left[\frac{-(c-1)(d-1)}{cd-c-d} + \frac{r}{d} + \frac{s}{c}\right] - \left(\frac{2}{c}+\frac{2}{d}\right) &= 2\left[-\left(1+\frac{1}{(c-1)(d-1)-1}\right) + \frac{r}{d} + \frac{s}{c}\right] - \left(\frac{2}{c}+\frac{2}{d}\right)\\
        &\geq 2\left[-1 - \frac{1}{(c-1)(d-1)-1} + 1 + 1\right] - \left(\frac{2}{c}+\frac{2}{d}\right)\\% \text{ (since $r\geq d, s\geq c$)}\\
        &= 2 - \frac{2}{(c-1)(d-1) - 1} - \left(\frac{2}{c}+\frac{2}{d}\right).
    \end{align*}
    Note that this expression is nonnegative precisely when
\[\frac{1}{(c-1)(d-1)-1}+\frac{1}{c}+\frac{1}{d}\leq 1.\]
    %This is positive when $(c-1)(d-1) > 7$.
    An easy check shows that if $c = 2$ then this inequality holds for $d\geq 6$.
    
    If $c\geq 3$, then this will hold so long as $d\geq 3$, with equality when $d = 3$. Notice that for $G = K_{2,3}, K_{2,4}, K_{2,5}$, we have $\kemeny_e(G) < \kemeny_{nb}(G)$, and if $G = K_{3,3}$ then $\kemeny_e(G) = \kemeny_{nb}(G)$. These can be shown to be the only $(c,d)$-biregular graphs in which $\kemeny_e(G) < \kemeny_{nb}(G)$.
    
    \end{proof}

    \begin{theorem}
    Let $G$ be a $(c,d)$-biregular graph which is not $K_{2,3}, K_{2,4}, K_{2,5},$ or $K_{3,3}$. Then
    \[
    1 - \frac{c+d}{cd} \leq \frac{\kemeny_{nb}(G)}{\kemeny_e(G)} < 1.
    \]
    %(the upper bound holds if $c=2, d\geq 6$ or $c,d\geq3.$ \red{Similar question as in thm 3.4})
    \end{theorem}
    \begin{proof}
    First note that the upper bound is a restatement of Theorem \ref{thm:BipartiteDifference}.
    
    To prove the lower bound, the following substitutions will be useful.
%    \begin{align*}
 %       n &= r+s & m &= cr = ds = \frac{cr+ds}{2}.
%    \end{align*}
    Now consider
    \begin{align*}
        \frac{\kemeny_{nb}(G)}{\kemeny_e(G)} &= \frac{2(m-n+1)(c-1)(d-1)}{\kemeny_e(G)[(c-1)(d-1)-1]} + \frac{2(r-s)(d-1)}{\kemeny_e(G)d} + \frac{1}{2\kemeny_e(G)} + \frac{2(s-1)}{\kemeny_e(G)} + \frac{cd-c-d}{cd}\\
        &\qquad + \frac{cd-c-d}{\kemeny_e(G)cd}[-2m+n-\frac{1}{2}-r+s]\\
        &= \frac{2(m-n+1)(c-1)(d-1)}{\kemeny_e(G)[(c-1)(d-1)-1]} + \frac{2(r-s)(d-1)}{\kemeny_e(G)d} + \frac{1}{2\kemeny_e(G)} + \frac{2(s-1)}{\kemeny_e(G)} + \frac{cd-c-d}{cd}\\
        &\qquad - \frac{cd-c-d}{\kemeny_e(G)cd}[2s(d-1)+\frac{1}{2}]
    \end{align*}
    Notice that $\kemeny_e(G)>0$. Then the lower bound holds if and only if
    % \[
    % 0\leq\kemeny_e(G)\left(\frac{\kemeny_{nb}(G)}{\kemeny_e(G)} -  1+\frac{c+d}{cd}\right).
    % \]
    \[
\kemeny_e(G)\left(\frac{\kemeny_{nb}(G)}{\kemeny_e(G)} -  1+\frac{c+d}{cd}\right)\geq 0.
    \]
    Considering this expression, we can rewrite it as
    \begin{eqnarray}
    \kemeny_e(G)\left(\frac{\kemeny_{nb}(G)}{\kemeny_e(G)} -  1+\frac{c+d}{cd}\right) & = & \frac{2(m-n+1)(c-1)(d-1)}{(c-1)(d-1)-1} + \frac{2(r-s)(d-1)}{d} \notag\\
    & & \qquad  +\frac{1}{2} + 2(s-1) - \frac{cd-c-d}{cd}\left(2s(d-1)+\frac12\right)\notag\\
    & = & \frac{2(m-n+1)(c-1)(d-1)}{(c-1)(d-1)-1} + \frac{2(d-1)r}{d} +\frac{c+d}{2cd}\notag\\
    & & \qquad  + s\left(4-\frac{2}{c}-\frac{2}{s}\right) - 2m + 2r ~~ \text{ (since $m = cr$.)} \label{eq:big_exp}
    \end{eqnarray}
    %\begin{align*}
     %   0&\leq\kemeny_e(G)\left(\frac{\kemeny_{nb}(G)}{\kemeny_e(G)} -  1+\frac{c+d}{cd}\right)\\
      %  &=\frac{2(m-n+1)(c-1)(d-1)}{(c-1)(d-1)-1} + \frac{2(r-s)(d-1)}{d} + \frac{1}{2} + 2(s-1) - \frac{cd-c-d}{cd}[2s(d-1)+\frac12]\\
        %&=\frac{2(m-n+1)(c-1)(d-1)}{(c-1)(d-1)-1} + 2(d-1)\left[\frac{r-s}{d} - \frac{(cd-c-d)s}{cd}\right] + \frac{1}{2}\left[1 - \frac{cd-c-d}{cd}\right] + 2(s-1)\\
        %&=\frac{2(m-n+1)(c-1)(d-1)}{(c-1)(d-1)-1} + 2(d-1)\left[\frac{cr-cds+ds}{cd}\right] + \frac{c+d}{2cd} + 2(s-1)\\
        %&=\frac{2(m-n+1)(c-1)(d-1)}{(c-1)(d-1)-1} + 2(d-1)\left[\frac{c(r-ds)}{cd} + \frac{s}{c}\right] + \frac{c+d}{2cd}+2(s-1)\\
        %&=\frac{2(m-n+1)(c-1)(d-1)}{(c-1)(d-1)-1} + 2(d-1)\left[\frac{r}{d} + \frac{s}{c} - s\right] + \frac{c+d}{2cd} + 2(s-1)\\
        %&=\frac{2(m-n+1)(c-1)(d-1)}{(c-1)(d-1)-1} + 2(d-1)\left[\frac{r}{d} + \frac{s}{c}\right] + \frac{c+d}{2cd} + 2s-2 - 2(d-1)s\\
        %&=\frac{2(m-n+1)(c-1)(d-1)}{(c-1)(d-1)-1} + 2(d-1)\left[\frac rd + \frac sc\right]  + \frac{c+d}{2cd} + 4s - 2(m+1)\\
        %&=\frac{2(m-n+1)(c-1)(d-1)}{(c-1)(d-1)-1} + \frac{2(d-1)r}{d} + \frac{2ds}{c} - \frac{2s}{c} + \frac{c+d}{2cd} + 4s - 2m-2\\
        %&=\frac{2(m-n+1)(c-1)(d-1)}{(c-1)(d-1)-1} + \frac{2(d-1)r}{d} + \frac{2m}{c} - \frac{2s}{c} + \frac{c+d}{2cd} + 4s - 2m-2\\
        %&=\frac{2(m-n+1)(c-1)(d-1)}{(c-1)(d-1)-1} + \frac{2(d-1)r}{d} + \frac{c+d}{2cd} + s\left(4-\frac{2}{c}-\frac{2}{s}\right) - \frac{2m(c-1)}{c}\\
       % &=\frac{2(m-n+1)(c-1)(d-1)}{(c-1)(d-1)-1} + \frac{2(d-1)r}{d} + \frac{c+d}{2cd} + s\left(4-\frac{2}{c}-\frac{2}{s}\right) - 2m + 2r \text{ (since $m = cr$.)}
    %\end{align*}
    Here, notice that $4-\tfrac{2}{c}-\tfrac{2}{s} \geq 2$ since both $c\geq 2$ and $s\geq2$. Then we get $
    s\left(4 - \frac2c - \frac2s\right) \geq 2s.$
    Recall also that $n = r + s$. Applying these observations to \eqref{eq:big_exp}, we now have 
    \begin{eqnarray*}
    \kemeny_e(G)\left(\frac{\kemeny_{nb}(G)}{\kemeny_e(G)} -  1+\frac{c+d}{cd}\right) & \geq & \frac{2(m-n+1)(c-1)(d-1)}{(c-1)(d-1)-1} + \frac{2(d-1)r}{d} + \frac{c+d}{2cd} + 2n - 2m\\
    & = & \frac{2(m+1)(c-1)(d-1)}{(c-1)(d-1)-1} - \frac{2n}{(c-1)(d-1)-1} + \frac{2(d-1)r}{d} + \frac{c+d}{2cd} - 2m\\
    & = & 2(m+1) + \frac{2(m+1) - 2n}{(c-1)(d-1)-1} +\frac{2(d-1)r}{d} +\frac{c+d}{cd} - 2m\\
    & = &  2 + \frac{2(m-n+1)}{(c-1)(d-1)-1} + \frac{2(d-1)r}{d} + \frac{c+d}{2cd}.
    \end{eqnarray*}
    % \begin{align*}
    %     &\frac{2(m-n+1)(c-1)(d-1)}{(c-1)(d-1)-1} + \frac{2(d-1)r}{d} + \frac{c+d}{2cd} + 2n - 2m\\
    %     =&\:\frac{2(m+1)(c-1)(d-1)}{(c-1)(d-1)-1} - \frac{2n}{(c-1)(d-1)-1} + \frac{2(d-1)r}{d} + \frac{c+d}{2cd} - 2m\\
    %     =&\:2(m+1) + \frac{2(m+1) - 2n}{(c-1)(d-1)-1} +\frac{2(d-1)r}{d} +\frac{c+d}{cd} - 2m\\
    %     =&\: 2 + \frac{2(m-n+1)}{(c-1)(d-1)-1} + \frac{2(d-1)r}{d} + \frac{c+d}{2cd}.
    % \end{align*}
    This expression is positive, and so it holds that
    \begin{equation*}
        \frac{\kemeny_{nb}(G)}{\kemeny_e(G)} \geq 1 - \frac{c+d}{cd}.
    \end{equation*}
    \end{proof}
      Note that when $c = d$ this is the same lower bound for the ratio obtained in the regular graph case above, Theorem \ref{thm:kemratio}.
    
\section{Cycle Barbells}
As these expressions for Kemeny's constant in the edge space depend on both the number of vertices and the number of edges, the question arises, ``What comparisons are meaningful comparisons?". One method for ensuring meaningful comparisons of edge Kemeny's constant between graphs is to compare only graphs that have the same number of both vertices and edges. A natural starting place for where the non-backtracking Kemeny's constant will be defined is the family of graphs with $n$ vertices and $n+1$ edges.

Using SageMath, we compute the values of the edge and non-backtracking Kemeny's constants for all graphs with minimum degree 2 on $n$ vertices and $n+1$ edges, up to order $n=20$.  These computations suggest that, for both $\kemeny_{nb}$ and $\kemeny_e$, the largest Kemeny's constant with these constraints occurs for graphs that we will call ``cycle barbells." 

We give a definition of cycle barbells and then proceed to calculate $\kemeny_{nb}$ and $\kemeny_e$ for these graphs.

\begin{definition}
    The \emph{cycle barbell} $G = CB(k,a,b) = C_a \oplus P_k \oplus C_b$ is the 1-sum of an
    $a$-cycle, a path on $k$ vertices, and a $b$-cycle. Note $|V(G)| = a+b+k-2$ and
    $|E(G)| = a+b+k-1$.
\end{definition}

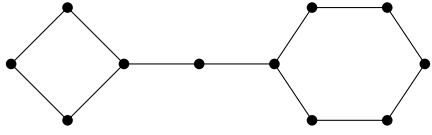
\begin{figure}[H]
    \centering
    \begin{tikzpicture}
    \tikzstyle{every node}=[circle,  fill=black, minimum width=4pt, inner sep=1pt]
        \draw{
        %C_4
        (-.5,.5)node{}--(0.25,1.25)node{}--(1,0.5)node{}--(0.25,-.250)node{}--(-.5,.5)
        %P_3
        (1,0.5)--(2,0.5)node{}--(3,0.5)node{}
        %C_6
        (3,0.5)--(3.5,1.25)node{}--(4.5,1.25)node{}--(5,0.5)node{}--(4.5,-.250)node{}--(3.5,-.250)node{}--(3,0.5)
        };
    \end{tikzpicture}
    \caption{The graph CB(3,4,6).}
    \label{fig:cb_graph}
\end{figure}

%Since Kemeny's constant and moments are known for cycles and paths, we can use Lemma \ref{thm:kem1sepformula} to compute the vertex Kemeny's constant of a cycle barbell.

\begin{theorem}\label{thm:VkemCB}
For a cycle barbell $G = CB(k,a,b)$, the vertex Kemeny's constant is given by
  \begin{align*}
  \kemeny_v(G) &=\frac{1}{a+b+k-1}\cdot \left[\frac{(a+1)(a-1)}{6}(a+2(b+k-1)) + \frac{(b+1)(b-1)}{6}(b+2(a+k-1))\right.\\
  &\left.+ (a+b)(k-1)^2+\frac{(k-1)(2k^2-4k+3)}{6}+2ab(k-1)\right].
  \end{align*}
  
  %\red{SEE IF SIMPLIFIES ANY NICER WAY}
  %Let $E = m+n+k-1.$
  %\begin{align*}
      %\kemeny_v(G) %&=\frac{1}{6E}\cdot\left[E^3-E^2+E+k^3-2k^2+4+3(m+a)(k^2-1)-k((m+a)+3)((m+a)+1)\right.\\
   %   &\left.-m^2(a-2)-a^2(m-2)+4am(2k-1) \right]
  %\end{align*}
  %\red{NOTE: Not really that nice but as far as optimization goes, this will keep a fair amount constant...}
\end{theorem}
\begin{proof}
Since the cycle barbell is a $1$-sum of two cycles and a path, we can use  Lemma \ref{thm:kem1sepformula} to give the result, relying on known expressions for the resistance distances in paths and cycles. In particular, using methods from Chapter 10 of \cite{bapat2010graphs}, it can be shown that in a cycle $C_n$ the resistance distance between two vertices $i,j$ is $r_{C_n}(i,j) = \frac{d(i,j)(n-d(i,j))}{n}$, where $d(i,j)$ is the shortest path distance from $i$ to $j$. In addition, for trees, $r(i,j) = d(i,j)$. From here the vertex Kemeny's constant and moment expressions are easily calculated using Lemma \ref{thm:kemeny} and Definition \ref{def:moment}, and combined using Lemma \ref{thm:kem1sepformula} to give the result. 
\end{proof}

The edge Kemeny's constant for a cycle barbell follows easily from Theorem \ref{thm:VkemCB} and Theorem \ref{thm:kemSRWedge}.

\begin{corollary}
For a cycle barbell $G = CB(k,a,b)$, the edge Kemeny's constant is given by
  \begin{align*}
  \kemeny_e(G) &=\frac{1}{a+b+k-1}\cdot \left[\frac{(a+1)(a-1)}{6}(a+2(b+k-1)) + \frac{(b+1)(b-1)}{6}(b+2(a+k-1))\right.\\
  &\left.+ (a+b)(k-1)^2+\frac{(k-1)(2k^2-4k+3)}{6}+2ab(k-1)\right]+a+b+k.
  \end{align*}
\end{corollary}

In order to find the non-backtracking Kemeny's constant for the cycle barbells we will find the characteristic polynomial of the non-backtracking transition probability matrix, and apply  Lemma \ref{lem:KemenyFromCharPoly} to calculate $\kemeny_{nb}(G)$.
For $G = CB(k,a,b)$, this matrix is given by 
\[P_{nb}(G) = 
\begin{bmatrix}
    \widehat C_a & 0 & 0 & 0 & 0 & \frac12S_a\\
    0 & \widehat C_a & 0 & 0 & 0 & \frac12S_a\\
    0 & 0 & \widehat C_b & 0 & \frac12S_b & 0\\
    0 & 0 & 0 & \widehat C_b & \frac12S_b & 0\\
    \frac12R_a & \frac12R_a & 0 & 0 & J_{k-1}(0) & 0\\
    0 & 0 & \frac12R_b & \frac12R_b & 0 & J_{k-1}(0)
\end{bmatrix}
\]
where $S_a$ is the $a\times (k-1)$ matrix that is all 0's except for a 1 in the bottom left entry, $R_a$ is the $(k-1)\times a$ matrix that is all 0's except for a 1 in the bottom left entry, $J_{k-1}(0)$ is a $(k-1)\times (k-1)$ Jordan block with 0 on the diagonal, and $\widehat{C}_a$ is an $a\times a$ matrix with 1's on the super diagonal, $1/2$ in the bottom left entry, and 0's everywhere else.  
%Define $\widehat D_a$ to be the $a\times a$ matrix as seen below.
%\begin{align*}
%    \widehat D_a &= \begin{bmatrix}
%        0 & 1 & 0&\cdots &0\\
%        0 & 0 & 1 & \cdots &0\\
%        \vdots & & \ddots & \ddots &  \vdots\\
%        &&&\ddots&1\\
%        \frac12 &0 & \cdots & &0
%    \end{bmatrix} 
%\end{align*}

 %(\red{Note}, in the $m=n$ case the characteristic polynomial does indeed reduce as we expect. I have not yet checked if the eigenvectors defined here do as well, so that could be worth looking into)
 %\red{CAREFUL: When relabeling $m$ to $a$ and $n$ to $b$, the following section also used, $a,b$ for other things. Relabel to $\alpha,\beta$ maybe?}
 
\begin{lemma}\label{lem:CBdiffsize-NBRW-charpoly}
Let $G = CB(k,a,b)$. Then $P_{nb}(G)$ has characteristic polynomial
\[
p(t) = (2t^a-1)(2t^b-1)[(2t^a-1)(2t^b-1)t^{2(k-1)}-1].
\]
\end{lemma}
\begin{proof}
  %We will guess eigenvectors and show that they are indeed eigenvectors. 
  %\red{(In my notes I assumed WLOG $m\geq n$ but I'm not sure if that's necessary or not. Look through again.)} That shouldn't be necessary
  
%   \[P_{nb} = 
% \begin{bmatrix}
%     \widehat D_a & 0 & 0 & 0 & 0 & \frac12S_a\\
%     0 & \widehat D_a & 0 & 0 & 0 & \frac12S_a\\
%     0 & 0 & \widehat D_b & 0 & \frac12S_b & 0\\
%     0 & 0 & 0 & \widehat D_b & \frac12S_b & 0\\
%     \frac12R_a & \frac12R_a & 0 & 0 & J_{k-1}(0) & 0\\
%     0 & 0 & \frac12R_b & \frac12R_b & 0 & J_{k-1}(0)
% \end{bmatrix}
% \]

%The eigenvalues and eigenvectors from $2t^a-1$ and $2t^b-1$ come from a similar proof used in \cite{glover2021non}.
We will determine eigenvectors of $P_{nb}$.  Suppose $2\lambda^a - 1 = 0$. Let $x = [\lambda\,\, \lambda^2\, \cdots\, \lambda^a]^T$. Then computation reveals that $[x^T\, -x^T\, 0\,\, 0\,\, 0]^T$ is an eigenvector for $P_{nb}$ corresponding to $\lambda$.

Suppose $2\lambda^b - 1 = 0$. A similar construction of $x$ will give $[0\,\, 0\,\, x^T\, -x^T\, 0]^T$ as an eigenvector for $P_{nb}$.

Suppose that $\lambda$ is a solution to $(2t^a-1)(2t^b-1)t^{2(k-1)}-1 = 0$.
%WLOG $a\geq b$.
Let $x, y, f, g, \alpha, \beta$ be as follows.
\begin{align*}
    x &= [\lambda^{k-1} \,\hdots\, \lambda^{a+k-2}]^T & y&=\beta[\lambda^k\,\hdots\,\lambda^{b+k-1}]^T\\
    f &= [1 \,\lambda\,\hdots\,\lambda^{k-2}]^T & g &= \alpha[\lambda^k\, \hdots\, \lambda^{2(k-1)}]^T\\
    \alpha &= \frac{2\lambda^a-1}{\lambda} & \beta &= \frac{1}{2\lambda^{b+k}-\lambda^k}.
\end{align*}
Note that $\beta$ is well-defined since if $2\lambda^{b+k} - \lambda^k = 0$ then $\lambda$ cannot be a root of $(2t^a-1)(2t^b-1)t^{2(k-1)}-1.$

Computation reveals that
\[
P_{nb}\begin{bmatrix}
    x\\
    x\\
    y\\
    y\\
    f\\
    g
\end{bmatrix} = \lambda\begin{bmatrix}
    x\\
    x\\
    y\\
    y\\
    f\\
    g
\end{bmatrix}.
\]
It is easily verified that this forms a complete set of linearly independent eigenvectors.
\end{proof}

\begin{theorem}\label{thm:kemNBRWBarbell}
The non-backtracking Kemeny's constant for a cycle barbell $G = CB(k,a,b)$ is given by
%\[
%\kemeny_{nb}(G) = \frac{5(m+n)^2+2mn-9(m+n)+2k^2+k(8(m+n)-5)+3}{2(m+n+k-1)}.
%\]
%POTENTIAL REWRITE
\[
\kemeny_{nb}(G) = \frac{2(a+b+k-1)^2+3(a+b)^2+2ab+4(a+b)(k-1)-(a+b+k-1)}{2(a+b+k-1)}.
\]
%POTENTIAL REWRITE
%\[
%\kemeny_{nb}(G) = \frac{5(m+n+k-1)^2 + (m+n+k-1)-(3k-1)(k-1)+2(mn-mk-nk)}{2(m+n+k-1)}
%\]
\end{theorem}
\begin{proof}
This follows from computation using Lemma \ref{lem:CBdiffsize-NBRW-charpoly} and Lemma \ref{lem:KemenyFromCharPoly}.

%Characteristic polynomial of normalized Laplacian is $p_{\mathcal{L}}(x) = p(1-x) = \cdots +c_2x^2 + c_1x$.
%$c_2 = p_{\mathcal{L}}''(0)/2$ and $c_1 = p_{\mathcal{L}}'(0).$
%Kemeny is $\kemeny(G) = -c_2/c_1.$
%(page 303 of my "nb Examples" notebook is where I did this I think)
\end{proof}

In these next results we show which barbells are the maximizers for the variants of Kemeny's constant among barbells on $n$ vertices. It is especially interesting that the non-backtracking Kemeny's constant has a different maximizer than the edge Kemeny's constant.

\begin{theorem}\label{thm:maxEdgeBarbell}
The edge Kemeny's constant for a cycle barbell on a fixed number of vertices is maximized when $a=b=3$, and the path has all the remaining vertices; that is, the extremal graph is $CB(n-4,3,3)$.  Moreover, \[\kemeny_e(CB(n-4,3,3))=\frac{2n^3+12n^2-51n+101}{6(n+1)}.\] 
\end{theorem}
\begin{proof}
Let $G = CB(k,a,b)$ and $m = |E(G)|$ and $n = |V(G)|$.
Since for a fixed number of vertices $2m-n$ is a constant, we can optimize the vertex Kemeny's constant. We begin by noticing that if $m = a+b+k-1$ then we can rewrite the expression in Theorem \ref{thm:VkemCB} as
\begin{align*}
      \kemeny_v(G) &=\frac{1}{6m}\cdot\left[m^3-m^2+m+k^3-2k^2+4+3(a+b)(k^2-1)-k((a+b)+3)((a+b)+1)\right.\\
      &\left.-a^2(b-2)-b^2(a-2)+4ab(2k-1) \right]
  \end{align*}
  Using this expression we see the only nonconstant terms are
\[
k^3 - 2k^2 + 3(a+b)(k^2-1) - k(a+b+3)(a+b+1) - a^2(b-2) - b^2(a-2) + 4ab(2k-1).
\]
Now suppose that $k$ is fixed. Then since $n$ is fixed, $a+b$ is also fixed. Thus the only nonconstant terms are
\begin{equation}\label{eq:edgeNonConstantTerms}
4ab(2k-1) -a^2(b-2) - b^2(a-2).
\end{equation}
Say $a+b = R$. Then we can reduce this expression to a function of a single variable by replacing $a = R-b$ and we obtain
\begin{align*}
    &4b(R-b)(2k-1) - (R-b)^2(b-2) - b^2(R-b-2)=b^2(R-8(k-1)) - bR(R-8(k-1)) + 2R^2.%-4b^2(2k-2)+4bR(2k-2)+2R^2
\end{align*}
As a function of $b$ this expression has a critical value at $b = R/2$ (hence when $a=b$). This is a maximum if $R<8(k-1)$, a minimum if $R>8(k-1)$, and is constant with value $2R^2$ if $R = 8(k-1)$.

If $R>8(k-1)$ then Equation (\ref{eq:edgeNonConstantTerms}) will be largest when $a$ (or $b$) is as small as possible (i.e. $b=3$). In this case one shows that Equation (\ref{eq:edgeNonConstantTerms}) is less than $2R^2$.
If $R<8(k-1)$ one can show that Equation (\ref{eq:edgeNonConstantTerms}) is greater than $2R^2$.  Thus we see for fixed $k$, the cycle barbell will have largest vertex Kemeny's constant when $a=b$ and $R<8(k-1)$.

%Now a single variable optimization tells us this function is maximized at $b=R/2$ and so for a fixed value of $k$, the cycle barbell will have largest Kemeny's constant when $a=b$.

Now let $a=b$. We will optimize letting $k$ vary. The expression of interest in this case becomes
\[
k^3 - 2k^2 + 6b(k^2-1) - k(2b+3)(2b+1) - 2b^3 + 4b^2 + 4b^2(2k-1).
\]
%This is readily seen to be largest when $k$ is as large as possible.
This is seen to be strictly increasing for $0\leq k\leq n-4$.

Therefore, the cycle barbell on $n$ vertices with maximal vertex Kemeny's constant---and hence maximal edge Kemeny's constant---is when $a=b=3$ and the path is as long as can be.
\end{proof}

We remark that in the proof above, when the order of the graph is fixed and $a+b=8(k-1)$ for some fixed $k$, the expression for the edge Kemeny's constant (and thus also the vertex Kemeny's constant) is independent of the choice of $a$ and $b$.  Thus surprisingly, for that particular length of path, it does not matter how balanced the two cycles are.
%%\red{Insert remark about $R=8(k-1)$ case.}

\begin{theorem}\label{thm:maxNBRWBarbell}
The non-backtracking Kemeny's constant for a cycle barbell on a fixed number of vertices $n$ is maximized at $CB(2, \lceil n/2\rceil, \lfloor n/2\rfloor)$.  Moreover, for $n$ even we have 
\[\kemeny_{nb}(CB(2,n/2,n/2))=\frac{11n^2+14n+2}{4(n+1)}\]
and for $n$ odd
\[\kemeny_{nb}(CB(2,(n+1)/2,(n-1)/2))=\frac{11n^2+14n+1}{4(n+1)}.\]
\end{theorem}
\begin{proof}
Let $G = CB(k,a,b)$.
For a fixed number of vertices $n$, a cycle barbell also has a fixed number of edges $m$. Thus the only non constant terms of the expression for $\kemeny_{nb}$  given by Theorem \ref{thm:kemNBRWBarbell} will be
\[
3(a+b)^2+2ab+4(a+b)(k-1).
\]
We first show that for fixed $k$, this is maximized when $a=b$. This is easily seen since if $k$ fixed and $n$ is fixed, then so is $a+b$. Thus this comes down to optimizing $2ab$ which is largest when $a=b$.

Now consider $a=b$ and constant $n$. We find the maximum value here. The expression of interest becomes
\[
3(2a)^2 + 2a^2+4(2a)(k-1) = 14a^2+8a(k-1).
\]
Notice that $2a+k-2 = n$ and so $k = n+2-2a$. Substituting this in yields
\[
-2a^2+8(n+1)a.
\]
Simple analysis shows us that this will attain it's maximum at $a=2(n+1)$, but $a\leq n/2$ and since the expression is increasing for $a< 2(n+1)$ the barbell will be maximized when $a$ is as large as possible. In the case that $n$ is even this is the graph $CB(2,n/2,n/2)$.  % $C_{n/2}\oplus P_2\oplus C_{n/2}.$ 
%For odd $n$ it is easy to see that $C_{(n+1)/2}\oplus P_2\oplus C_{(n-1)/2}$ is maximal. 
For odd $n$, computation shows that $\kemeny_{nb}(CB(2,a+1,a)) > \kemeny_{nb}(CB(3,a,a))$.
\end{proof}

\begin{figure}[h]
\begin{center}
    \includegraphics[scale=0.75]{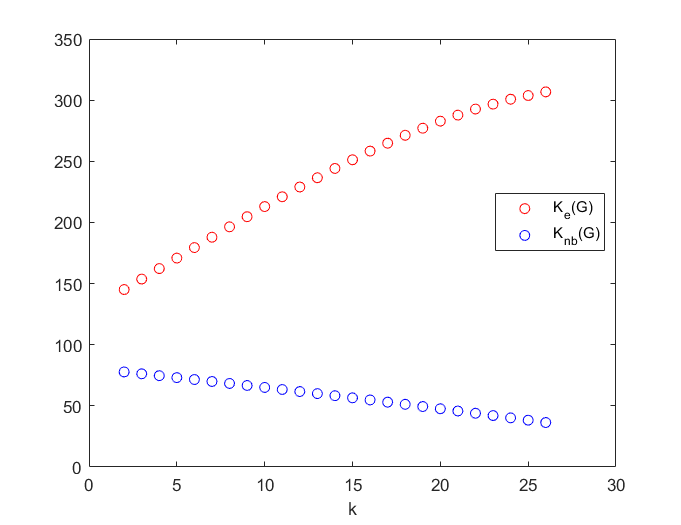}
\end{center}
\caption{A comparison of the values of $\kemeny_e(G)$ and $\kemeny_{nb}(G)$, where $G$ is a cycle barbell of order $n=30$, with $k$ varying from $2$ to $26$, and $a=b=\tfrac{1}{2}(n-k+2)$.}
\label{fig:cb_comp}
\end{figure}

We end with some discussion of open questions and avenues of research.  Theorems \ref{thm:maxEdgeBarbell} and \ref{thm:maxNBRWBarbell} exhibit interesting differences in behavior between a simple random walk and a non-backtracking random walk. For cycle barbells, the simple random walk Kemeny's constant is largest when there was a large path and small cycles, whereas in the non-backtracking random walk Kemeny's constant was largest when there was a small path with large cycles. Also note that the edge Kemeny's constant is an order of magnitude larger than the non-backtracking Kemeny's constant even when both are compared at $G = CB(2,n/2,n/2)$ (the maximizer for the non-backtracking walk, and minimizer for the simple walk); see Figure~\ref{fig:cb_comp}. In particular $\kemeny_{nb}(G) = O(n)$ and $\kemeny_e(G) = O(n^2)$. This suggests that, while long paths will tend to lead to a large Kemeny's constant for the simple walk, large cycles make more of a difference for the non-backtracking walk. It would be interesting to further investigate more generally what graph properties lead to large or small simple walk Kemeny's constant versus a large non-backtracking walk Kemeny's constant.  

Note that from Theorem \ref{thm:kemratio}, for regular graphs, the simple walk and non-backtracking walk Kemeny's constants have the same order of magnitude.  It is an interesting open question to determine for what graphs these orders of magnitude will be the same, and for what graphs they are different, and by how much they can differ.  Moreover, it is known that for the simple walk Kemeny's constant on the vertices, Kemeny's constant is at most on the order of $O(n^3)$ where $n$ is the number of vertices, and there are examples where this order of magnitude is achieved (see \cite{breen2019computing}).  In all examples from this work, the largest non-backtracking Kemeny's constant that we have seen is on the order of $O(n^2)$ (but again, the comparison based on size of the graph is a more subtle matter since the state space of the Markov chain is now the number of directed edges).  It would be of interest to determine if this is the largest possible order of magnitude.

Finally, in nearly all results from this paper, the non-backtracking Kemeny's constant is smaller than the simple edge Kemeny's constant.  The only exceptions to this have only a few vertices.  Indeed, we have done computations on all connected graphs with minimum degree at least 2 that are not cycles on up to 10 vertices.  We have found that for $n=4$ vertices there are 2 graphs with $\kemeny_{nb}(G)\geq\kemeny_e(G)$, on $n=5$ vertices there are 10 graphs with $\kemeny_{nb}(G)\geq\kemeny_e(G)$, on $n=6$ vertices there are 18 graphs with $\kemeny_{nb}(G)\geq\kemeny_e(G)$, on $n=7$ vertices there are 7 graphs with $\kemeny_{nb}(G)\geq\kemeny_e(G)$, on $n=8$ vertices there are 3 graphs with $\kemeny_{nb}(G)\geq\kemeny_e(G)$, and on $n=9$ and $n=10$ vertices, there are no graphs with $\kemeny_{nb}(G)\geq\kemeny_e(G)$. We conjecture that, for all graphs with sufficiently many vertices, the non-backtracking Kemeny's constant will be smaller. 

%\red{Sentence stating the number of exceptions to the conjecture for orders up to 10.}

%\begin{align*}
%\kemeny(G) &=\frac{1}{n+m+k-1}\cdot \left[\frac{(n+1)(n-1)}{6}(n+2(m+k-1)) + \frac{(m+1)(m-1)}{6}(m+2(n+k-1))\right.\\
%&\left.+ (n+m)(k-1)^2+\frac{(k-1)(2k^2-4k+3)}{6}+2mn(k-1)\right]
%\end{align*}

% Ke cycle barbell grows order magnitude bigger than Knb (even at  suspected min of Ke which is max of Knb)

%Note: Ke(CB) goes to Path graph

\bibliographystyle{plain}
\bibliography{sources}

\end{document}